\documentclass[12pt, leqno]{article}
\usepackage{amsgen, amsmath, amsfonts, amsthm, amssymb, enumerate,amscd}

\newtheorem{defn}{Definition}[section]
\newtheorem{prop}[defn]{Proposition}
\newtheorem{lem}[defn]{Lemma}
\newtheorem{thm}[defn]{Theorem}

\newcommand {\ZZ}{{\mathbb Z}}

\newcommand {\K}{{\mathcal K}}

\newcommand {\C}{{\mathbb C}}

\newcommand {\HH}{{\mathfrak  H}}

\newcommand {\I}{{\mathcal I}}

\def\deg{\operatorname{deg}}
\def\mod{\operatorname{mod}}

\title{Iterated integrals and higher order automorphic forms}
\author{Nikolaos Diamantis\\
Ramesh Sreekantan}
\date{}
\begin{document}

\maketitle

\bf Abstract. \rm Higher order automorphic forms have recently been
introduced to study important questions in number
theory and mathematical physics. We investigate the
connection between these
functions and Chen's iterated integrals. Then using
Chen's theory, we
prove a structure theorem for automorphic forms of all orders.
This allows us to define an analogue of a mixed Hodge structure
on a space of higher order automorphic forms.
\it



\section{Introduction}

In \cite{[G]} the notion of Eisenstein series with modular symbols
was introduced in order to study a new approach towards a
conjecture of Szpiro. This series is not invariant under the
action of the relevant group, but instead it satisfies a $4$-term
functional equation. Motivated by the applications of this
Eisenstein series (\cite{[G], [G1], [PR], [R]}),  and the form of
its functional equation which generalizes that of the classical
automorphic forms, the first author and others began the general
study of classes of functions satisfying equations of this type
(\cite{[CDO]}). Similar objects were defined and studied from a
different viewpoint by Kleban and Zagier (\cite{[KZ]}).

In this paper we complete the classification of a the space of
automorphic
forms of all orders and weights for a Fuchsian group of the first kind
$\Gamma$ without elliptic elements along the
lines of the classification of the space of second-order modular forms
proved in \cite{[CDO]}. It should be noted that the space classified here
is larger than that studied in \cite{[CDO]}. This was motivated by the
desire to study certain automorphic forms that do not seem to belong to
the smaller space.

For the classification we use Chen's theory of iterated integrals.
Although it is possible that there exists an alternative proof of
the classification that does not use iterated integrals, we wanted
to highlight this connection with the important theory of iterated
integrals. In this
approach, higher order modular forms can be loosely viewed as
antiderivatives of iterated integrals on the modular curve.
(See \cite{[HB]}, \cite{[M]} for other applications of iterated integrals
to
modular forms which however do not deal with higher orders).

Based on this classification, we impose a Mixed-Hodge-type structure on
the space of automorphic forms of all orders in the case of weight $0$.
Because of the infinite dimensionality of the quotients of the ``weight
filtration", this structure is not a standard Mixed Hodge Structure.

However the structure described here reflects in a very natural
way the algebraic structure of our space and it seems likely that
certain subspaces of these automorphic forms could have a usual
Mixed Hodge Structure.

\bf Acknowledgment. \rm The authors thank the referee for a very careful
reading of the paper and for many useful suggestions.

\section{Higher order automorphic forms}

Let $\Gamma \subset$PSL$_2(\mathbb Z)$
be a Fuchsian group of the first kind with parabolic elements acting in
the standard manner on the upper half-plane $\HH.$ We use the set of
generators of $\Gamma$ given by Fricke and Klein.  Specifically, if
$\Gamma \backslash \HH$ has genus $g,$ $r$ elliptic fixed points and $m$
cusps, then there are $2g$ hyperbolic elements $\gamma_1, \dots,
\gamma_{2g},$ $m$ parabolic elements $\gamma_{2g+1}, \dots, \gamma_{2g+m}$
and $r$ elliptic elements $\gamma_{2g+m+1}, \dots, \gamma_{2g+m+r}$
generating $\Gamma.$ Furthermore, these generators satisfy the $r+1$
relations: $$ [\gamma_1,\gamma_{g+1}]\dots [\gamma_g,\gamma_{2g}]
\gamma_{2g+1} \dots \gamma_{2g+m} \gamma_{2g+m+1} \dots \gamma_{2g+m+r}=1,
\, \, \gamma_j^{e_j}=1$$ for $2g+m+1 \le j \le 2g+m+r$ and integers $e_j
\ge 2$. Here $[a, b]$ denotes the commutator $aba^{-1}b^{-1}$ of $a$ and
$b$.

We set $Y(\Gamma)$ for the modular curve $\Gamma \backslash \HH$
and we consider the natural projection map $\pi: \HH \to Y(\Gamma).$
For a function $f$ on $\HH$ and an even integer $k$, set
$$(f|_k{\gamma})(z):=(cz+d)^{-k} f(\gamma z)$$
for $\gamma=\begin{pmatrix} a & b\\c & d \end{pmatrix}$ in
$\Gamma$. This defines an action of $\Gamma$ on the space of
complex functions on $\mathfrak H.$ We extend this action to $\mathbb
C[\Gamma]$
by linearity.

A classical automorphic form of weight $k$ for $\Gamma$ is a smooth
function $f$ of ``at most polynomial growth at the cusps" such that
$$f|_k{(\gamma-1)}=0$$
Let $J$ denote the augmentation ideal of the group ring which lies in the
exact sequence
$$0 \rightarrow J \rightarrow \ZZ[\Gamma]
\stackrel{\deg}{\rightarrow} \ZZ \rightarrow 0.$$
$J$ is generated by elements of the form $(\gamma-1)$ with
$\gamma$ in $\Gamma$, so we can define a classical \it automorphic form
\rm as a function of ``at most polynomial growth at the cusps" which is
annihilated by $J$ via $|_k$. We call holomorphic automorphic forms, \it
modular. \rm

With that in mind, we define an \it automorphic form $f$ of order
$s$ \rm to be a smooth function on $\mathfrak H$ such that
$$f|_k \delta=0 \text{ for all $\delta \in J^{s}$}.$$
Let $M_k^s(\Gamma)$ denote the automorphic forms of weight $k$ and
type $s$. From the definition, for a fixed $k$,  we have
$$M^0_k(\Gamma) \subseteq M^1_k(\Gamma) \subseteq
M^2_k(\Gamma)\dots  \subseteq M^s_k(\Gamma).$$
Classical automorphic forms are elements in $M_k^1(\Gamma)$ that
satisfy certain growth conditions at the cusps. Of course, there are
several variations of the definition of higher order modular
forms. See \cite{[CDO]}, \cite{[DKMO]} for a related discussion.

The first step towards the classification of automorphic forms of order
$s$ is
\begin{prop}
Let $$\psi: M_k^{s+1} \to \bigoplus ^{(2g+m-1)^s}_{i=1} M_k^1$$ be defined
by
$$\psi(f)=(f|_k(\gamma_{i_1}-1) \dots (\gamma_{i_s}-1))_{1 \le i_1,
\dots, i_s \le 2g+m-1}$$
Then,
$$0 \to M_k^s \hookrightarrow M_k^{s+1}
\xrightarrow{\psi}
\bigoplus_{i=1}^{(2g+m-1)^s} M_k^1 $$
is an exact sequence.
\label{exact}
\end{prop}

\begin{proof} To prove that $\ker(\psi) \subset M_k^s$, we first observe
that, if $f \in M_k^{s+1}$, then for each $\delta_1 \in J^m,$
$\delta_2 \in J^{s-m-1},$ we have
\begin{eqnarray}
f|_k\delta_1 (\gamma_1 \gamma_2-1)
\delta_2=f|_k\delta_1(\gamma_1-1)\delta_2+
f|_k\delta_1(\gamma_2-1)\delta_2.
\label{kerphi}
\end{eqnarray}
This follows from the observation that
\begin{equation}
\gamma_1 \gamma_2-
1=(\gamma_1-1)(\gamma_2-1)+(\gamma_1-1)+(\gamma_2-1).
\label{prod1}
\end{equation}
Now let $f$ be in $\ker(\psi)$. Using \eqref{kerphi} we observe
that to prove $f$ is in $M^s_{k}(\Gamma)$, it suffices to verify
that
$$f|_k(g_1-1) \dots (g_s-1)=0$$
for each $s$-tuple of {\em generators} $g_i$ of $\Gamma.$

From the definition of  $\psi$, $f|_k(g_1-1) \dots (g_{s}-1)=0$
for all the non-elliptic generators $g_i \in \{\gamma_1, \dots,
\gamma_{2g+m-1}\}$ of $\Gamma$.

Further, let $\gamma \in \Gamma$ be one of the elliptic generators
with $\gamma^{e}=1.$ As $\delta_1(\gamma^{l}-1)$ is in $J^{m+1}$, for
every $\delta_1 \in J^m,$ we have
$$f|_k\delta_1 \gamma^l(\gamma-1)\delta_2=f|_k\delta_1 (\gamma-1)\delta_2$$
for $\delta_1 \in J^m,$ $\delta_2 \in J^{s-m-1}$, $l=0, \dots
e-1$. Therefore,
$$ef|_k\delta_1 (\gamma-1)\delta_2=f|_k\delta_1(\sum_{l=0}^{e-1}\gamma^l)
(\gamma-1)\delta_2=f|_k\delta_1(\gamma^e-1)\delta_2=0$$
so $$f|_k\delta_1 (\gamma-1)\delta_2=0 \text{ for all } \delta_1
\in J^m,\delta_2 \in J^{s-m-1}.$$

Finally, using the relation between the generators, we can write
$$\gamma_{2g+m}=([\gamma_1,\gamma_{g+1}]\dots
[\gamma_g,\gamma_{2g}]\gamma_{2g+1} \dots \gamma_{2g+m-1})^{-1}
(\gamma_{2g+m+1} \dots \gamma_{2g+m+r})^{-1}$$
and, by (2), $\gamma_{2g+m+1}-1$ can be expressed in terms of the other
generators.
This implies that $f|_k(g_1-1) \dots
(g_{s}-1)=0$ for all $g_i$'s in the set of generators of $\Gamma$
and the middle term of the sequence is exact.
\end{proof}

The surjectivity of $\psi$ will be studied in Section 4. To this end, we
will need to define Chen's iterated integrals and to review their basic
properties.

\section{Iterated integrals}

Let $X$ be a smooth manifold.  Let $P(X)$ denote the space of
paths on $X$, namely piecewise smooth
$$\gamma:[0,1] \rightarrow X.$$
A function $\phi: P(X) \to \mathbb C$ is said to be a homotopy
functional if $\phi$ depends only on the homotopy class of
$\gamma$ relative to its endpoints,
that is, it defines a function on
$\Gamma=\pi_1(X, x_0)$, where
$x_0$ is a fixed point of $X$. Equivalently, it induces
an element of Hom$(\ZZ[\Gamma], \mathbb C).$

Let $w$ be a smooth 1-form on $X$. The map
$$\gamma \rightarrow \int_{\gamma} w=\int_{0}^{1} f(t)dt$$
where
$\gamma^*(w)=f(t)dt$
defines  a function  on  $P(X)$.
This defines an element of Hom$(\ZZ[\Gamma],\mathbb C)$ if and only if
$w$ is closed. Hence this only detects elements of $\Gamma$
visible in the homology of $X$ - it vanishes on $J^2$ ($J$ denotes the
augmentation ideal of $\mathbb Z[\pi_1(X, x_0)]$ here).

The iterated integrals studied by Chen (e.g. \cite{[C]}) detect
more elements of the group ring. Specifically,
suppose that $w_1, w_2,...,w_r \in E^1(X)$, where
$E^1(X)$ denotes the space of smooth 1-forms on $X$.
We will write $$w_1 \dots w_s =w_1
\otimes \dots \otimes w_s  \in \bigotimes ^s E^1(X)$$ and
call it a ``product" of the $w_i$'s. We set $w_1 \dots
w_s=1$, when $s=0.$

If $\gamma$ is a path on $X$, we set
\begin{equation}
\int_{\gamma} w_1 w_2...w_r=
{\int ...\int}_{0\leq t_1 \leq t_2 ...\leq t_r \leq 1}
f_1(t_1)f_2(t_2)...f_r(t_r) dt_1dt_2...dt_r,
\end{equation}
where $\gamma^*(w_i)=f_i(t)dt$.
This defines  a function on the space of paths of $X$  which will
be denoted by  $\int w_1...w_r$ and  is called a
iterated line integral of  length $r$.   A linear combination  of
such  functions is called an {\em iterated integral} and  its
length is the length of the longest line integral. However, it is
not necessarily a homotopy functional.

Let $B_s(X)$ denote the space of iterated integrals of
length $\leq s$. If $\I$ is in $B_s(X)$ and $\alpha \in P(X)$ we denote
the evaluation map by $\langle \I,\alpha \rangle$. We
extend it to all $1$-chains by linearity.

The next theorem states that in some cases an iterated integral
can be modified to be a homotopy functional.

\begin{thm} [Chen](\cite{[C]}, Section 3.) Let $X$ be a connected, smooth
manifold with $H^2(X)=0$ and let $w_1, \dots w_s$ be closed
$1$-forms on $X$.
Then there is an $\I \in B_s(X)$ which is a homotopy
functional and a $\K \in B_{s-1}(X)$ satisfying
$$\langle \I, \alpha \rangle=\int_{\alpha} w_{1} \dots
w_{s}+ \langle \K, \alpha \rangle,$$
for each path $\alpha \in P(X)$.
\label{ccc}
\end{thm}
\begin{proof} An example of an $\I$ satisfying these conditions
can be constructed using an \it extended defining system for a Massey
product of \rm $w_1, \dots, w_s$: Fix smooth $1$-forms $w_{12}$,
$w_{23}, \dots, w_{123}, \dots$ such that
\begin{eqnarray*} &w_{1} \wedge w_{2}+dw_{12}=0, \dots, w_{s-1} \wedge
w_{s}+dw_{(s-1)s}=0, \dots \\
&  w_{1} \wedge w_{23}+w_{12} \wedge
w_{3}+dw_{123}=0, \dots  \\
&w_{1} \wedge w_{2 \dots s}+w_{12} \wedge
w_{3 \dots s}+\dots+dw_{1 \dots s}=0. \\
\end{eqnarray*}
We then set
$$\I=\int u$$
where
\begin{eqnarray*} &u:=
w_{1} w_{2} \dots w_{s}+ w_{12} w_{3} \dots
w_{s}+w_{1} w_{23} \dots
w_{s}+\dots+ \\
& w_{123} w_{4} \dots w_{s}+ w_{1} w_{234} \dots
w_{s}+ \dots +w_{1 \dots s}.
\end{eqnarray*}
Note that $\int u - \int w_{1}w_{2}\dots w_{s}$ is of
length $< s$. The proof of the independence of path can be found
in \cite{[C]}, page 366.
\end{proof}
If  $w$ is a  1-form and $\alpha,\beta$  are two loops based
at $x_0$, then it is easy to see
$$\langle \int w,(\alpha-1)(\beta-1)\rangle=0$$
We will need the following lemmas. The first
lemma generalizes the above comment.
\begin{lem}(\cite{[HR]}, Lemma 2.10, Proposition 2.13)
Let $w_1, \dots w_r$ be smooth $1$-forms on $X$ and let
$\alpha=\prod_{i=1}^{s}(\alpha_i-1)$, where $\alpha_i$ are loops based at
$x_0$. Then
$$\langle \int w_1 \dots w_r,\alpha \rangle=\begin{cases} 0 & \text{ if }
r<s\\
\prod_{i=1}^s \int_{\alpha_i} w_i & \text{ if }
r=s.\end{cases}$$
\label{hain1}
\end{lem}
The second lemma describes what happens under composition of
paths.
\begin{lem}(\cite{[HR]}, Prop. 2.9)
Let $w_1, \dots w_s$ be smooth $1$-forms on $X$ and let $\alpha,
\beta$ be paths such that $\alpha(1)=\beta(0)$. Then
\begin{eqnarray*}
\langle \int w_1 \dots w_s,\alpha \beta \rangle= \langle \int w_1 \dots
w_s,\alpha \rangle+ \langle
\int w_1 \dots w_s, \beta \rangle\\+\sum_{j=1}^{s-1}\langle \int w_1 \dots
w_j,\alpha \rangle \langle \int w_{j+1} \dots w_s, \beta \rangle.
\end{eqnarray*}
\label{hain2}
\end{lem}

An application of Lemma \ref{hain2} and Theorem \ref{ccc} is the
following.
\begin{lem} Let $X$ be a connected, smooth manifold
with $H^2(X)=0$ and let $w_1, \dots w_s$ be closed $1$-forms on
$X$.
If $\{w_{12}, \dots, w_{(s-1) s}, \dots, w_{1 \dots s}\}$
is an extended defining system for a Massey product of $w_1, \dots
w_s$, we set \newline
\centerline{$u_j=w_{1} \dots w_{j}+ w_{12} w_{3} \dots
w_{j}+w_{1} w_{23} \dots
w_{j}+ \dots +w_{1 \dots j}$, \, \, $j=1, \dots s$
} \newline
\centerline{$u^j=w_{j+1} \dots w_{s}+ w_{(j+1) (j+2)} w_{j+3} \dots
w_{s}+w_{j+1} w_{(j+2) (j+3)} \dots
w_{s}+ \dots +w_{j+1 \dots s}$ \, \, $j=1, \dots s-1$.
} \newline
We also set $u:=u_s.$
Then for each pair of paths $\alpha$ and $\beta$ on $X$ with $\alpha(1)=
\beta(0)$, we have
\begin{equation}
\int_{\alpha \beta} u=\int_{\alpha} u+ \int_{\beta}u
+\sum_{j=1}^{s-1} \int_{\alpha} u_j \int_{\beta} u^j.
\end{equation}
\label{identity}
\end{lem}

\begin{proof}

From the construction of $u,$ all combinations of $1,2,\dots,s$
appear in $u$ (in this order) as indices of ``products" of $w$'s.
Applying Lemma \ref{hain2}, we can decompose the integral of each
individual ``product" as
a sum of products of iterated integrals on $\alpha$ and on $\beta$. Thus,
$$\int_{\alpha \beta} u=\sum_{j=0}^s \Big ( \sum_{v_j} \Big (
\sum_{v^j} \int_{\alpha} v_j \int_{\beta} v^j \Big ) \Big )$$
where $\sum_{v_j}$ (resp. $\sum_{v^j}$)
ranges over all ``products" $v_j$ (resp. $v^j$) with index sequences formed
by the integers $1, \dots, j$ (resp. $j+1, \dots s$).
(Here we have set $\int v_0=\int v^{s+1} \equiv 1.$) For example
$\sum_{v_3}\int_{\alpha}v_3=\int_{\alpha}( w_1w_2w_3+ w_1w_{23}+
w_{12}w_3+w_{123})$.

From the defining equations of $u_j$ and $u^j$,
$$\sum_{v_j} \int_{\alpha} v_j = \int_{\alpha} u_j \text{ and } \sum_{v^j}
\int_{\beta} v^j = \int_{\beta} u^j.$$
This proves the identity.
\end{proof}

\section{The classification of higher order automorphic forms}

We now restrict ourselves to the case when $X=Y(\Gamma),$ where
$\Gamma \subset$PSL$_2(\mathbb Z)$ is a Fuchsian group of the first kind
with parabolic elements, as in Section 2.
In addition, we assume that $\Gamma$ has no elliptic elements and hence
$\pi_1(X)=\Gamma$.
In this section, we complete the classification of the vector space
of automorphic forms of order $s$ and weight $k$ for such groups.
We first work
in the case of weight $0$ proving what amounts to a variation of
the general statement we will eventually prove.

Let $\tilde M^{s+1}_0=\tilde M^{s+1}_0(\Gamma)$ denote the space
of $f \in M^{s+1}_0=M^{s+1}_0(\Gamma)$ such that $f|_0 \delta \in
\mathbb C$ for every $\delta$ in $J^s$.
\begin{prop} The sequence
$$0 \to M_0^s \hookrightarrow \tilde M_0^{s+1}
\xrightarrow{\psi}
\bigoplus_{i=1}^{(2g+m-1)^s} \mathbb C \to 0$$
with $\psi$ defined as in Proposition \ref{exact} is exact.
\label{wt0}
\end{prop}
\begin{proof} The exactness of the middle term follows from Proposition
\ref{exact} because $\tilde M_0^{s+1} \subset M_0^{s+1}$. From that, we
observe that to prove surjectivity, it suffices to construct
$(2g+m-1)^s$ linearly independent elements of $\tilde
M_0^{s+1}/M_0^s.$

Let $\{f_1, \dots, f_g, g_1, \dots, g_{m-1} \}$ be a basis of the
space of holomorphic modular forms of weight $2$ for $\Gamma$ with
$f_i$ cuspidal and $g_i$ non-cuspidal. Let $\{w_i\}_{i=1}^g$ be
the differentials on $Y(\Gamma)$ corresponding to
$\{f_i(z)dz\}_{i=1}^g,$ $\{w_{i}\}_{i=g+1}^{2g}$ the differentials
corresponding to $\{\overline{ f_{i-g}(z)} d \bar z
\}_{i=g+1}^{2g}$ (where the bar stands for complex conjugation)
and $\{w_{2g+i}\}_{i=1}^{m-1}$ the differentials corresponding to
$\{g_i(z) dz \}_{i=1}^{m-1}$.

Then, all $w_i,$'s are closed, as are $w_i \wedge w_j$ ($i, j=1,
\dots, 2g+m-1$). Furthermore, since $Y(\Gamma)$ is non-compact,
from the Gysin exact sequence, we have
$$H^2(Y(\Gamma),\C)=0.$$
Hence we can apply Theorem \ref{ccc} to any selection of $s$ forms from
$\{w_i; i=1, \dots, 2g+m-1 \}$. Let $I=(i_1,i_2,\dots , i_s)$ be
any indexing
vector with elements in $\{1, \dots, 2g+m-1\}$ and let $\I_I,\K_I$ and
$u_I$ be induced by $\{w_{i_j}\}_{j=1}^s$ as in the proof of
Theorem \ref{ccc}. Let
$x_0$ be a point in the upper half-plane lying over a point, which
we will also denote by $x_0$ on the curve. If we let $\{x_0, b\}$
denote the image under $\pi$ of the line path from $x_0$ to $b$ in
$\mathfrak H$, then the function
\begin{equation}
F_{\I_I}(z):=F_{I}(z):=\langle \I_{I}, \{x_0, z\}  \rangle , \qquad z \in
\mathfrak H
\label{antiderivative}
\end{equation}
is well-defined and independent of the path from $x_0$ to $z$, since
$\pi$ maps homotopic paths to homotopic paths.

We will now show the
$F_I$'s are in $\tilde{M}_0^{s+1}(\Gamma)$ for each $I$.
In Lemma \ref{linindlem} we will further show that they are linearly
independent modulo $M_0^s(\Gamma)$. As there are $(2g+m-1)^s$ of them,
these two facts will suffice to prove Prop. \ref{wt0}.

We first use Lemma
\ref{identity} to compute $F_I|_0 \delta$, ($\delta
\in J^s$) and show that it is a constant. We use the notation
$\{a, b\}$ for the path $\{a, x_0\}$ followed by $\{x_0, b\}.$

\begin{lem} Let $I=(i_1,\dots,i_s)$ be an indexing vector
with elements in $\{1, \dots 2g+m-1\}$ and set $F_I(z)=\langle
\I_I,\{x_0,z\} \rangle$. Then, for any $\delta=\prod_{k=1}^{s}
(\gamma_k-1) \in J^s$,
$$(F_I|_0\delta)(z)=\prod_{k=1}^{s} \int_{\{z, \gamma_k z\}} w_{i_k}$$
In particular, since the $w_{i}$ correspond to classical modular
forms of weight $2$ and their conjugates, this expression is
independent of $z$ and so $F_I(z) \in \tilde M^{s+1}_0(\Gamma)$.
\label{mainlem}
\end{lem}
\begin{proof}
For every $\gamma \in \Gamma,$
\begin{equation}
F_{I}|_0 (\gamma-1)(z)=\langle \I_{I}, \{x_0, \gamma z\}
\rangle -\langle \I_{I}, \{x_0, z\} \rangle.
\nonumber
\end{equation}
Combining this with Lemma \ref{identity} with $\alpha=\{x_0, z\}$
and $\beta=\{z, \gamma z\}$  gives
\begin{equation}
F_{I}|_0(\gamma-1)(z)=\langle \I_{I}, \{z, \gamma z\} \rangle +
\sum_{j=1}^{s-1} \int_{\{x_0, z\}} u_{j} \int_{\{z, \gamma z\}}
u^j
\label{step2}
\end{equation}
Observe that, if $\gamma_1, \dots, \gamma_s$ are in $\Gamma$, we
have the following expressing generalizing \eqref{prod1},
\begin{equation}
\prod_{k=1}^s(\gamma_k-1)=(\gamma_1 \dots
\gamma_s-1)+\dots+(-1)^{s-1}(\gamma_1-1)\dots+(-1)^{s-1}(\gamma_s-1).
\nonumber
\end{equation}
Combining this with \eqref{step2}, we have
\begin{align}
F_{I}|_0(\gamma_1-1)\dots (\gamma_s-1)(z) &= F_{I}|_0((\gamma_1
\dots \gamma_s-1)+\dots+(-1)^{s-1}(\gamma_s-1))(z) \nonumber \\
&= \langle \I_{I}, \{z, \gamma_1 \dots \gamma_s z\} \rangle +\dots
+(-1)^{s-1}\langle \I_{I}, \{z, \gamma_s z\} \rangle \nonumber \\
&+ \sum_{j=1}^{s-1} \int_{\{x_0, z\}} u_j \Big ( \int_{\{z,
\gamma_1 \dots \gamma_s z\}} u^j+\dots +(-1)^{s-1}\int_{\{z,
\gamma_s z\}} u^j\Big ) \nonumber.
\end{align}
Next we observe that, in $\HH$, the path from $\gamma z$ to $\gamma \delta
z$ passing through $x_0$ is homotopic to $\gamma (z, \delta z),$
where $(z, \delta z)$ is the path from $z$ to $\delta z$ passing
through $x_0$. Since the image of $\gamma (z, \delta z)$ under
$\pi$ is $\{z, \delta z\},$ the loops $\{\gamma z, \gamma \delta
z\}$ and $\{ z, \delta z \}$ are homotopic in $X$ and hence $\{z, \gamma
\delta z \}$ is homotopic to $\{z, \gamma z \}\{z, \delta z \}.$
On the other hand, by the proof of Theorem \ref{ccc}, each $\int u^j$ is
homotopy
invariant. Therefore, by induction, for all $\gamma_1, \dots
\gamma_s \in \Gamma,$ we have
\begin{align}
F_{I}|_0(\gamma_1-1) \dots (\gamma_s-1)(z)= \langle \I_{I}, (\{z,
\gamma_1 z \}-1) \dots (\{z, \gamma_s z\}-1) \rangle
\nonumber
\\ + \left(
\sum_{j=1}^{s-1} \langle \int u_j, \{x_0, z\} \rangle \langle \int u^j,
(\{z, \gamma_1 z \}-1) \dots (\{z, \gamma_s z\}-1) \rangle \right)
\nonumber
\end{align}
Since the iterated integrals in the sum within the parenthesis are of
length $<s$ and since $\I_I \equiv \int \prod_{k=1}^s w_{i_k}$ up
to an iterated integral of length $<s$, by Lemma (\ref{hain1}) we
deduce
$$F_{I}|_0(\gamma_1-1) \dots (\gamma_s-1)(z)=
\int_{\{z, \gamma_1 z \}} w_{i_1}  \dots \int_{\{z, \gamma_s z
\}} w_{i_s}$$
\end{proof}

To complete the proof of Proposition \ref{wt0}
we show that the images of $F_{I}$ under the
natural projection
$$ \tilde M^{s+1}_0 \to \tilde M^{s+1}_0/M^{s}_0$$ are linearly
independent.

\begin{lem} Suppose there exist
complex numbers $k_{I}$ satisfying
$$\sum_I k_{I} F_{I} \in M_0^s(\Gamma)$$
where $I$ runs through the $(2g+m-1)^s$ possible $s$-tuples of
$\{1,\dots ,2g+m-1\}$. Then $k_I=0$ for all $I$.
\label{linindlem}
\end{lem}
\begin{proof} We proceed by induction. Suppose $s=1$ and that
there are complex numbers $k_r$ such that
$$\sum_{r=1}^{2g+m-1} k_{r} F_{r}  \in M_0^1(\Gamma)$$

Then, from Lemma \ref{mainlem}, for any $\gamma \in \Gamma$ we
have
$$\sum_{r=1}^{2g+m-1} k_{r} F_{r}|_0(\gamma-1)(z)=\sum_{r=1}
^{2g+m-1} k_{r} \int_{\{z, \gamma z\}} w_{r}=0$$
where $w_r$ are the (holomorphic or anti-holomorphic)
differential forms corresponding to the
basis of the space of weight $2$ modular forms for $\Gamma$
we fixed at the beginning of the proof of Proposition \ref{wt0}.
From the injectivity of the classical Eichler-Shimura
isomorphism, we have $k_{r}=0$ for all $r$.

Proceeding by induction, suppose there are complex numbers $k_{I}$
satisfying
$$\sum_I k_{I} F_{I} \in M_0^s $$
where $I$
runs through the $(2g+m-1)^s$
possible $s$-tuples $(i_1,\dots,i_s)$ of $\{1,\dots ,2g+m-1\}$. Then
$$\sum_I k_{I} F_{I} |_0(\gamma_1-1)\dots (\gamma_s-1)=0$$
for any $\gamma_1,\dots,\gamma_s$ in $\Gamma$.

Applying Lemma \ref{mainlem}, this is equivalent to
$$\sum_I k_I \int_{\{z, \gamma_1z \}} w_{i_1}\dots \int_{\{z, \gamma_s
z\}} w_{i_s}=0$$
or
\begin{multline*}
\sum_{i_s=1}^{2g+m-1} \left( \sum_{i_{s-1}=1}^{2g+m-1}\dots
\sum_{i_1=1}^{2g+m-1} k_{I} \int_{\{z, \gamma_1 z\}} w_{i_1}
\dots
\int_{\{z, \gamma_{s-1}z\}} w_{i_{s-1}} \right) \int_{\{z, \gamma_s
z \}}w_{i_s} \\
=\sum_{i_s=1}^{2g+m-1} A_{i_s} \int_{\{z, \gamma_{s}z \}}
w_{i_s}=0.
\end{multline*}
From the Eichler-Shimura isomorphism once again, we have that
$A_{i_s}=0$ for all $i_s$. For a fixed $i_s$, $A_{i_s}$ is
$$\sum_{I'} k_{I'\cup\{i_s\}} \int_{\{z, \gamma_1z \}} w_{i_1} \dots
\int_{\{z, \gamma_{s-1}z \}} w_{i_{s-1}}$$
where $I'$ runs through all possible $(s-1)$-tuples. By induction,
$$k_{I'\cup\{i_s\}}=0.$$
As this is true for all $i_s$, $k_I=0$ for all $I$ and the $F_I$
are linearly independent modulo $M_0^s(\Gamma)$.
\end{proof}

This completes the proof of Proposition \ref{wt0}.
\end{proof}
The following classification theorem is an
application of Proposition \ref{wt0}.
\begin{thm} Let $$\psi: M_k^{s+1} \to \bigoplus ^{(2g+m-1)^s}_{i=1}
M_k^1$$ be defined
by
$$\psi(f)=(f|_k(\gamma_{i_1}-1) \dots (\gamma_{i_s}-1))_{1 \le i_1,
\dots, i_s \le 2g+m-1}$$
Then,
$$0 \to M_k^s \hookrightarrow M_k^{s+1}
\xrightarrow{\psi}
\bigoplus_{i=1}^{(2g+m-1)^s} M_k^1 \to 0$$
is an exact sequence.
\label{rightexact}
\end{thm}

\begin{proof} In view of Prop. \ref{exact}, the only part that needs to be
proved is the surjectivity of $\psi.$ Let
$$(f_{I})_{I} \in \bigoplus_{i=1}^{(2g+m-1)^s} M_k^1$$
with $I$ ranging through the $(2g+m-1)^s$ possible $s$-tuples of
$\{1, \dots, 2g+m-1\}$.
We will show that there is a $F \in M_k^{s+1}$ such that
$\psi(F)=(f_I)_I$ or equivalently,
$$ F|_k(\gamma_{i_1}-1)\dots(\gamma_{i_s}-1)=f_{I}$$
for all $s$-tuples $I=\{i_i,\dots,i_s\}$.

From the surjectivity part of  Theorem \ref{wt0}, for each
$s$-tuple of integers $L=(l_1, \dots l_s)$ with $l_j  \in \{1,
\dots, 2g+m-1\}$, there is a $\Lambda_{L} \in \tilde M_0^{s+1}$
such that
$$\Lambda_{L}|_0(\gamma_{i_1}-1)\dots
(\gamma_{i_s}-1)=\delta_{I}^L$$
for any $s$-tuple $I$, where $\delta_I^L$ is the Kronecker delta
function of the $s$-tuple, namely $\prod_k \delta_{i_k}^{l_k}$.

An easy computation then shows that
$$F=\sum_{L}f_{L} \Lambda_{L}$$
is in $M_k^{s+1}$ and satisfies the desired equality.
\end{proof}

We shall finally use the last two propositions to show that the
direct sum of all $M_0^s$ can be endowed with a structure which in
some respects is very similar to a Mixed Hodge Structure.

\begin{thm} Set
$$V=\bigoplus_{s=1}^{\infty} M_0^s.$$ There exists \newline
$\bullet$ a increasing filtration
$W_{\bullet}$ on $V$ and \newline
$\bullet$ a decreasing filtration $F^{\bullet}$ on $V$\newline
such that, for each $m \in \mathbb N,$  the filtration
on $Gr_m(W^{\bullet})=W_m/W_{m-1}$ defined by
$$F^p(Gr_m(W^{\bullet}))=W_m \cap F^p /W_{m-1} \cap F^p$$
satisfies
$$F^p(Gr_m(W^{\bullet})) \oplus \bar F^{m-p+1}(Gr_m(W^{\bullet})) \cong
Gr_m(W^{\bullet})$$
\label{mhs}
\end{thm}

\begin{proof} It suffices to construct two filtrations $W_{\bullet}$
and $F^{\bullet}$ satisfying the property: \newline
\it If $w \in W_m$ and $p, q$ are integers such that $p+q=m+1,$ then there
exist $w_1 \in F^p \cap W_m,$ $w_2 \in F^q \cap W_m$ and $w_0 \in W_{m-1}$
such that $w=w_1+{\bar w}_2+w_0$. Moreover, this decomposition is unique
modulo $W_{m-1}.$ \rm

We define $W_{\bullet}$ setting $W_m=M_0^{m+1}.$

To define $F^{\bullet}$ we first consider the set $\mathfrak F$ of
all functions that are induced
by modular forms of weight $2$
and their conjugates as in \eqref{antiderivative}.
Specifically, $\mathfrak F$ contains all functions $f:\HH \to \mathbb C$
with
the property that
$$f(z)= \langle \I, \{x_0, z\}\rangle,
\qquad \text{for all} \quad
z \in \HH,$$
for a homotopy functional $\I \in B_s(X)$ such that
$$\I \equiv \int w_1 \dots w_s \mod B_{s-1}(X)
$$
for some $w_i$ corresponding to weight $2$ modular forms and their
conjugates.
We do not require that the modular forms
inducing $f$ belong to the basis fixed in the proof of Prop. 4.1.

We next consider the set $\mathfrak F^p$ of all $F \in
\mathfrak F$ induced by forms $w_i$ among which at least $p$ are
holomorphic. $F^p$ is then defined as the space
generated by products of the form $f \cdot F$,  with $f$ in
$M_0^1$ and $F \in \mathfrak F^p$, i.e.
$$F^p:=\langle f \cdot F; f \in M_0^1, F \in \mathfrak F^p \rangle$$

Now, Theorem \ref{rightexact} implies that each $F \in W_{m}=M_0^{m+1}$
can be expressed as a sum of functions of the form $f_L \cdot \Lambda_L $,
where $L$ is an indexing vector of length $m$ and $f_L \in M_0^1$.
By the construction of $\Lambda_L,$
each of them can be expressed, uniquely modulo a function in $M_0^{m}$, as
a linear combination of the $F_I$'s constructed in Prop. 4.1. Here $I$
is an indexing set of length $m$.
Therefore, $F$ can be written in the form
\begin{equation}
F=\sum_I f_I F_I+F_0
\label{mhs1}
\end{equation}
for some $F_0 \in W_{m-1}=M^{m}_0$ and
$f_I \in M_0^1$.

Let $p, q$ be positive integers such that $p+q=m+1$ and let $ F_I$
be one of the functions in the right-hand side of \eqref{mhs1}.
If the set of $1$-forms inducing $F_I$ contains less than $p$ holomorphic
forms, then it contains at least $q=m-p+1$ anti-holomorphic ones.
In addition, by the construction of ${\I}_I$'s and $F_I$'s,
$\bar F_I \in \mathfrak F$. Hence $\bar F_I$ is induced by at least $q$
holomorphic forms and $\overline{f_I F_I} \in F^q$. This completes the
proof.
\end{proof}

As mentioned above, this structure is not quite a MHS mainly because the
$Gr_m(W^{\bullet})$'s are not finite dimensional. It is also not
clear that the definition is functorial. However, it appears that
there is an interesting
subspace which has a  genuine mixed Hodge structure. We
will discuss this in a future paper.

\bf Addresses: \rm Nikolaos Diamantis: School of Mathematical Sciences,
University of Nottingham, Nottingham, NG7 2RD, United Kingdom \\
nikolaos.diamantis@nottingham.ac.uk \\[0.4cm]
Ramesh Sreekantan: School of Mathematics, Tata Institute of Fundamental Research,
Dr Homi Bhabha Rd, Bombay 400 005, India \\
ramesh@math.tifr.res.in

\end{document}